\documentclass[ECP]{ejpecp}

\usepackage{microtype}
\usepackage[shortlabels]{enumitem}
\usepackage{theoremref}
\usepackage{tikz}
\usepackage{mathrsfs}
\usepackage{hyperref}
\usepackage{comment}

\SHORTTITLE{The critical density for the frog model is the degree of the tree}

\TITLE{The critical density for the frog model is the degree of the tree\thanks{The first author
was partially supported by NSF grant DMS-1401479.}}

\AUTHORS{%
  Tobias Johnson\footnote{New York University.
    \EMAIL{tobias.johnson@nyu.edu} }
  \and 
  Matthew Junge \footnote{Duke University. \EMAIL{jungem@math.duke.edu} }
  }


\AMSSUBJ{60K35; 60J80; 60J10} 

\SUBMITTED{July 28, 2016}
\ACCEPTED{November 15, 2016}

\VOLUME{21}
\YEAR{2016}
\PAPERNUM{82}
\DOI{16-ECP29}

\ABSTRACT{The frog model on the rooted $d$-ary tree changes from transient to recurrent as the number of frogs per site is increased. We prove that the location of this transition is on the same order as the degree of the tree.}

\renewcommand{\emptyset}{\varnothing}
\newcommand{\rootvertex}{\varnothing}

\newcommand{\E}{\mathbf{E}}
\renewcommand{\P}{\mathbf{P}}
\newcommand{\Aa}{\mathcal{A}}
\newcommand{\Bb}{\mathcal{B}}
\newcommand{\Uu}{\mathcal{U}}

\newcommand{\TT}{\mathbb{T}}
\newcommand{\ZZ}{\mathbb{Z}}

\newcommand{\f}{\frac}
\newcommand{\ind}[1]{\mathbf 1 \{ #1 \}}

\newcommand{\RR}{\mathbb{R}}

\newcommand{\rS}{T}
\makeatletter
\newcommand*\proc{{\mathpalette\bigcdot@{.7}}}
\newcommand*\bigcdot@[2]{\mathbin{\vcenter{\hbox{\scalebox{#2}{$\m@th#1\bullet$}}}}}
\makeatother
\DeclareMathOperator{\Poi}{Poi}

\DeclareMathOperator{\Bin}{Bin}

\newtheorem{thm}{Theorem}
\newtheorem{prop}[thm]{Proposition}

\theoremstyle{definition}

\newcommand{\stprec}{\preceq}
\newcommand{\stsucc}{\succeq}

\newcommand{\kd}{\lceil d/c \rceil}
\newcommand{\fC}{h_{C,c}}
\newcommand{\Vset}{\mathcal{V}}

\begin{document}


\section{Introduction}

The \emph{frog model} in its most general form is an interacting particle system taking place
on a countable collection of vertices, typically a graph, with a designated vertex $\emptyset$ that
we call the root. The process is given by a pair
$(\eta, S)$ of particle counts and paths. Initially, there is one active particle at the root, and
$\eta(v)$ sleeping particles at each nonroot vertex~$v$. When activated, the $i$th particle
starting at vertex~$v$  moves according to the path $S_\proc(v,i)$, with $S_0(v,i)$ assumed to be $v$.
Whenever an active particle moves to a new vertex, it activates all sleeping particles there.
We continue the common practice of referring to the particles as frogs.

Let $\TT_d$ denote the infinite $d$-ary tree rooted at $\emptyset$; that is, $\TT_d$ is the tree
in which $\emptyset$ has degree~$d$ and all other vertices have degree~$d+1$. Our topic is the frog
model on $\TT_d$ where $(\eta(v))_{v\neq\emptyset}$ is an i.i.d.-$\Poi(\mu)$ collection of random variables,
$(S_\proc(v,i))_{v,i}$ is a collection of independent simple random walk paths, and the two collections
are independent. We abbreviate this as the frog model on $\TT_d$ with i.i.d.-$\Poi(\mu)$ frogs per site.

We call a realization of the frog model \emph{transient} if the root is visited finitely many times
by frogs and \emph{recurrent} if it is visited infinitely often. Recurrence versus
transience is perhaps the most basic question for the frog model. It has been studied on $\ZZ^d$
under a variety of initial conditions and frog paths
\cite{telcs1999,random_recurrence,recurrence,new_drift,01frog,rosenberg}. In \cite{HJJ1,HJJ2,JJ3_order},
we address the question on $d$-ary trees. For more background material on the frog model,
see the introduction to \cite{JJ3_order}.

We show in \cite{HJJ2} that the frog model on a $d$-ary tree with
i.i.d.-$\Poi(\mu)$ frogs per vertex undergoes a phase transition between transience and recurrence
as $\mu$ grows. In more detail, for each $d\geq 2$ there exists a critical value $\mu_c(d)$
such that the model is almost surely transient for $\mu<\mu_c(d)$ and is almost surely recurrent for
$\mu>\mu_c(d)$. The proof in \cite{HJJ2} shows that for some constants $C,C'>0$,
we have $ Cd \leq \mu_c(d) \leq C'd\log d$.
In this paper, we sharpen this result by removing the $\log d$ factor from the upper bound,
determining the order of $\mu_c(d)$ up to constant factors.

\begin{thm}\thlabel{thm:critical}
 For all sufficiently large $d$, it holds that
$.24 d \leq \mu_c(d) \leq 2.28d$.
\end{thm}

\thref{thm:critical} can be extended to initial distributions other than Poisson using
the results of \cite{JJ3_order}. For example, the frog model on a $d$-ary tree
with deterministically many frogs per site has a critical threshold of the same order; see
\cite[Corollary~4]{JJ3_order} for details.\looseness=1

An accurate description of the transition threshold on trees is especially relevant given that the frog model on $\mathbb Z^d$ behaves rather differently. A transition still occurs, but it does so at a decaying density of frogs (see \cite[Theorem 1.1]{random_recurrence}).
A natural next step would be to investigate this problem for irregular trees. Perhaps the phase
transition is on the same order as the branching number of the tree (see \cite[Section~1.2]{LyonsPeres})?
This is completely speculative, as the mere existence of a phase transition for the frog model is
unknown even on a Galton--Watson tree.

Our other goal for this paper
is to present as simply as possible the argument for existence of a recurrence phase on trees.
While the transience phase is fairly easy to establish (see \cite[Proposition~15]{HJJ2}),
the recurrence phase is more difficult. Indeed, the question of recurrence on $\TT_2$
with one sleeping frog per site was first posed in \cite{phasetree} and was only recently answered
in the affirmative \cite{HJJ1}. It remains open to determine if the one-per-site frog model
is recurrent on $\TT_3$ and $\TT_4$.
Taking advantage of some technical improvements we have made since \cite{HJJ1,HJJ2},
we give a streamlined proof of recurrence on $\TT_2$
with i.i.d.-$\Poi(\mu)$ frogs per site (see \thref{prop:binary.recurrence}). Using
\cite[Corollary~3]{JJ3_order}, this result also implies recurrence on
$\TT_2$ with two sleeping frogs per site. This is weaker than the result in \cite{HJJ1}, but the
proof is much simpler.\looseness=1

\subsection*{Ideas of the proof}
As in \cite{HJJ1,HJJ2}, our proof of recurrence is based on
recursion and bootstrapping. To set this up, we first show that it is enough to establish recurrence
for a frog model whose paths are stopped non-backtracking walks,
which we call the \emph{self-similar frog model}.
Let $V$ be the number of visits to the root in this process. A self-similarity
 yields a relation between $V$ and a collection of independent copies of $V$. Such relations are called recursive distributional equations (see \cite{AB} for further discussion).\looseness=1


In the bootstrap part of the argument, we assume that $V$ is stochastically larger than $\Poi(\lambda)$
for some $\lambda\geq 0$. We then analyze the recursive distributional equation to show that $V$
is in fact stochastically larger than $\Poi(\lambda+\epsilon)$.
Iterating this argument starting at $\lambda=0$, we show that $V$ is larger than $\Poi(\epsilon)$,
then larger than $\Poi(2\epsilon)$, and so on, with the conclusion that $V=\infty$ a.s.
Here and in \cite{HJJ2}, this argument uses the standard stochastic order,
while in \cite{HJJ1} it uses a more exotic stochastic order.

The bootstrap phase of our argument is more elaborate than in \cite{HJJ2}, where
we could only establish $\mu_c(d)=O(d\log d)$. We obtain a better upper bound in this paper
because the version of the self-similar frog model here better approximates the actual frog model. Put more simply, we are able to capture the contributions of more frogs.
The simplification that allows us to handle the extra complexity is our use of
\cite[Theorem~3.1(b)]{MSH}, a simple criterion for determining when a Poisson distribution
stochastically dominates a Poisson mixture. This allows us to avoid the more difficult coupling argument
used in \cite{HJJ2}.

\section{A criterion for stochastic dominance}\label{sec:orders}

Given two probability measures $\nu$ and $\nu'$ on the extended real numbers,
we say that $\nu$ is \emph{stochastically smaller} than $\nu'$ if
\begin{align*}
  \nu\bigl((t,\infty]\bigr) \leq \nu'\bigl((t,\infty]\bigr)
\end{align*}
for all $t\in\RR$. We denote this relationship by $\nu\stprec\nu'$.
If $X\sim \nu$ and $Y\sim\nu'$, we also write $X\stprec Y$, $X\stprec\nu'$, and
$\nu\stprec Y$ all to mean the same thing.
An alternate characterization of stochastic dominance is that $X\stprec Y$ if and only
if there exists a coupling of $X$ and $Y$ such that $X\leq Y$ a.s.

\thref{thm:dominance.criterion} provides a necessary and sufficient condition for a Poisson mixture to
dominate a Poisson distribution, which we will need in Section~\ref{sec:critical}.
We reproduce the proof given in \cite{MSH}
for our readers' convenience. See also \cite{Yu} for a more general result.
\begin{lemma}[{\cite[Lemma~3.1(b)]{MSH}}]\thlabel{lem:concavity}
  For any positive integer $n$, the function
  \begin{align*}
    h_n(x) = x\sum_{k=0}^n \frac{(-\log x)^k}{k!}
  \end{align*}
  is increasing and concave on $(0,1]$.
\end{lemma}
\begin{proof}
  We compute
  \begin{align*}
    h_n'(x) &= \sum_{k=0}^n \frac{(-\log x)^k}{k!} -\sum_{k=1}^n \frac{(-\log x)^{k-1}}{(k-1)!}
       = \frac{(-\log x)^n}{n!},
  \end{align*}
  which is positive and decreasing on $(0,1]$, showing that $h_n(x)$
  is increasing and concave.
\end{proof}

\begin{thm}[{\cite[Theorem~3.1(b)]{MSH}}]\thlabel{thm:dominance.criterion}
  Let $X\sim\Poi(\lambda)$, and let $Y\sim\Poi(U)$ for some nonnegative random variable~$U$.
  Then the following are equivalent:
  \begin{enumerate}[(i)]
    \item $X\stprec Y$, \label{i:st}
    \item $\P[X=0]\geq \P[Y=0]$, and \label{i:zero.prob}
    \item $\lambda\leq -\log\E e^{-U}$. \label{i.log}
  \end{enumerate}
\end{thm}
\begin{proof}
  Conditions~\ref{i:zero.prob} and~\ref{i.log} are just restatements
  of each other, since $\P[X=0]=e^{-\lambda}$ and $\P[Y=0]=\E e^{-U}$.
  Condition~\ref{i:st} implies \ref{i:zero.prob}
  by the definition of stochastic dominance.
  It remains to prove that \ref{i.log}
  implies \ref{i:st}.
  It suffices to show that $\P[Y\leq n]\leq \P[X\leq n]$ for all nonnegative
  integers $n$. We compute
  \begin{align*}
    \P[Y\leq n] &=\sum_{k=0}^n \E\Biggl[\frac{e^{-U} U^k}{k!}\Biggr]=
      \E\Biggl[\zeta\sum_{k=0}^n \frac{(-\log \zeta)^k}{k!}\Biggr] = \E h_n(\zeta),
  \end{align*}
  where $\zeta=e^{-U}$. Our assumption is that $\E\zeta=\E e^{-U}\leq e^{-\lambda}$.
  By Lemma~\ref{lem:concavity}, the function $h_n(x)$ is increasing and concave
  on $(0,1]$, where $\zeta$ takes values. Thus
  \begin{align*}
    \E h_n(\zeta)&\leq h_n(\E\zeta)\leq h_n\bigl(e^{-\lambda}) = \sum_{k=0}^n \frac{e^{-\lambda}\lambda^k}{k!}
      =\P[X\leq n],
  \end{align*}
  where we use that $h_n$ is concave to apply Jensen's inequality in the first step,
  and we use that $h_n$ is increasing in the second step.
\end{proof}

\section{Critical parameters for $d$-ary trees}\label{sec:critical}

Our argument breaks down into two parts described in the introduction.
In Section~\ref{sec:recall}, we define the self-similar frog model and define $V$ as the number
of visits to the root in this process. Then, we determine a recursive distributional equation satisfied by
the law of $V$ (see \thref{lem:fixed}).
The ideas in this section can be found in \cite{HJJ1,HJJ2}, but they take some work to extract
in the form we need.
Though we do our best to avoid duplicating material, when in doubt we have
opted for comprehensibility over efficiency.

Next comes the bootstrap portion of the argument.
In Section~\ref{sec:simplest}, we use the set-up of Section~\ref{sec:recall} to carry this out in case
of the binary tree, giving a short proof of the existence of a recurrence phase.
In Section~\ref{sec:bootstrap} we give a more complex version of this argument
proving recurrence on the $d$-ary tree for $\mu=\Omega(d)$.

\subsection{The bootstrapping set-up} \label{sec:recall}


\subsubsection{The self-similar frog model}

Many basic features of the frog model depend only on the range of each frog. This yields rather nice abelian and monotonicity properties. For example, the total number of visits to the root is unaffected by the order frogs wake up in and the rate they reveal vertices in their ranges.
Also, trimming the range of frogs can only reduce the number of visits to the root.
Applying this observation in combination with the coupling characterization of stochastic dominance,
we note the following fact. Let $r(\eta, S)$ be the number of visits to the root in the frog
model $(\eta,S)$.

\begin{fact} \thlabel{fact:mono}
Consider a collection of frog paths $S = \bigl(S_\proc(v,i)\bigr)_{v \in G, i \geq 1}$ on a graph~$G$.
Suppose that another collection of paths $ \tilde  {S}$ can be coupled with $S$
such that for all $i$ and $v$, the range of $\tilde S_\proc(v,i)$ is a subset
of the range of $S_\proc(v,i)$. Then $r(\eta,  \tilde {S}) \stprec r(\eta,S)$.
\end{fact}

From now on, let $S=(S_\proc(v,i),\,v\in\TT_d,i\geq 1)$ denote a collection of independent
simple random walks with $S_\proc(v,i)$ started at $v$, and let the components of
$\eta=(\eta(v))_{v\in\TT_d}$ be i.i.d.-$\Poi(\mu)$, independent of $S$.
The first step in studying the frog model $(\eta,S)$
will be to replace $S$ by a collection of paths $\rS$ to obtain $(\eta,\rS)$, which we call the
\emph{self-similar frog model} in reference to a useful property described in \thref{fact:same}.

We define $T$ in two steps. First, let $S'=(S'_\proc(v,i),\,v\in\TT_d,i\geq 1)$ denote a collection
of independent random non-backtracking walks stopped at $\emptyset$.
In more detail, call a random walk a \emph{simple random non-backtracking walk} on an arbitrary graph
if it chooses from its neighbors uniformly for its first step, and then in all subsequent steps it chooses
uniformly from its current neighbors except the one it just arrived from. We define
$S'_\proc(v,i)$ to be a simple non-backtracking random walk stopped on arrival at $\emptyset$.
The walks $S'_\proc(v,i)$ and $S_\proc(v,i)$ can be coupled
so that the range of the first is a subset of the range of the second by making $S'_\proc(v,i)$ a stopped,
loop-erased version of $S_\proc(v,i)$. This is proved in detail in \cite[Proposition~7]{HJJ1}.

Now we construct $T$ as a modification of $S'$. Each path $T_\proc(v,i)$ will be a stopped
version of $S'_\proc(v,i)$. Let $v$ be a nonroot vertex in $\TT_d$ with parent $u$. Suppose that
$v$ is visited in the frog model $(S', \eta)$ for the first time at time $j$, necessarily by one or
more frogs moving from $u$ to $v$.
Select one of these visiting frogs arbitrarily, and stop all of the other ones. (Observe that it
is irrelevant which frog is allowed to continue, so long as one views frogs as indistinguishable.)
If any frogs move from $u$ to $v$ at subsequent times, stop them at $v$ as well.
Do this for all vertices~$v\in\TT_d$, and
let $T$ be the resulting collection of stopped walks. As the range of each $T_\proc(v,i)$ is a subset
of the range of $S'_\proc(v,i)$, the following fact (also noted in \cite[Proposition~7]{HJJ1})
follows:
\begin{fact}\thlabel{fact:coupling}
  There is a coupling of $S$ and $T$ so that the range of each $T_\proc(v,i)$
  is a subset of the range of $S_\proc(v,i)$.
\end{fact}
By \thref{fact:mono,fact:coupling}, we have $r(\eta,T)\stprec r(\eta,S)$.
We will now work exclusively with the self-similar frog model, $(\eta, T)$,
and prove recurrence for it with sufficiently large $\mu$.
Unlike all other frog models considered in this paper, the frog paths $T$ are not independent
of each other nor of $\eta$, because one frog's motion in $(\eta,T)$ can cause another frog to be stopped.
This is the only form of dependence, however, and frogs that have not been stopped move independently
of each other. So, it is not a serious obstacle.

Let $V=r(\eta,T)$.
Next, we discuss a self-similarity property of the model and its consequences for $V$.
For any vertex $v\in\TT_d$, let $\TT_d(v)$ denote the subtree made up of $v$ and its descendants.
We call $\TT_d(v)$ \emph{activated} in the self-similar frog model if $v$ is ever
visited. Let $u$ be the parent of $v$.
By our construction of $T$, if $\TT_d(v)$ is activated, then there is a unique frog
that moves from $u$ to $v$, entering $\TT_d(v)$ and then never leaving it.
The frog model viewed starting from the time of activation only
at vertices $\{u\}\cup\TT_d(v)$ then looks identical to the original
self-similar frog model viewed on $\{\emptyset\}\cup\TT_d(\emptyset')$. This yields the following
fact, proved in more detail in \cite[Proposition~6]{HJJ1}.

\begin{fact}\thlabel{fact:same}
  Let $V'$ be the number of frogs that move from $v$ to its parent $u$ in the self-similar frog model.
  The distribution of $V'$ conditional on $\TT_d(v)$ being activated
  is identical to the distribution of $V$.
\end{fact}

The following observation shows that once a subtree $\TT_d(v)$ is activated, the random variable
$V'$ defined in the above fact is independent of the frog model outside of $\TT_d(v)$.
\begin{fact}\thlabel{fact:independence}
  Let $V'$ be defined as in \thref{fact:same}. Conditional on $\TT_d(v)$ being activated, $V'$ depends only
  on the path of the activator and on $\{T_\proc(w,i), \eta(w)\colon \,w\in\TT_d(v),i\geq 1\}$.
\end{fact}

We will use \thref{fact:same,fact:independence} to express $V$ recursively in terms of independent copies of itself,
an idea expressed in Figure~\ref{fig:1}.
This relation will be given in terms of an operator we define next.

\begin{figure}
  \begin{center}
    \begin{tikzpicture}[xscale=3,yscale=2,vert/.style={circle,fill,inner sep=0,
    minimum size=0.125cm,draw},
    active/.style={rectangle, fill, inner sep=0, minimum size=0.25cm,draw},
    sleeping/.style={rectangle, inner sep=0, minimum size=0.25cm,draw}]

\path (0,0) node[active,label=below:$\rootvertex'$] (root') {}
            (-1,.2) node[vert,label=below:$\rootvertex$] (root) {}
            (1,-0.8) node[active,label=below:$v_1$] (v1) {}
            (1,-0.15) node[sleeping,label=below:$v_2$] (v2) {}
            (1,.8) node[sleeping,label=above:$v_d$] (vd) {}
            (2,1) node (vd1) {}
            (2,.6) node (vd2) {}
            (2,.05) node (v21) {}
            (2,-.35) node (v22) {}
            (2,-.6) node (v11) {}
            (2,-1) node (v12) {}
            (1,.4) node {$\vdots$};

\path       (1.6,.8) node[label = right:$\TT_d(v_d)$] (tvd) {}
            (1.6,-.8) node[label = right:$\TT_d(v_1)$] (tv1) {}
            (1.6,-.15) node[label = right:$\TT_d(v_2)$] (tv2) {}
            ;

      \draw [thick, <-, shorten >= 0.1cm, shorten <= .15cm] (root) to[out=50,in=100]  node[black, above] {$V$} (root');

\draw      [thick, <-, shorten >= 0.25cm, shorten <= .6cm] (root') to[out=-80,in=-150]  node[above] {$V_1$} (v1)
      ;

\draw      [thick, dashed, <-, shorten >= 0.1cm, shorten <= .3cm] (root') to[out=20,in=100]  node[above] {$V_2$} (v2)
      ;

\draw      [thick, dashed, <-, shorten >= 0.5cm, shorten <= .4cm] (root') to[out=80,in=100]  node[below] {$V_d$} (vd)
      ;

      \draw[thick] (root) -- (root') -- (v1)
            (root')--(v2)
            (root')--(vd);

      \draw[thick] (vd)--(vd1)
                      (vd) -- (vd2)
                      (v2) -- (v21)
                      (v2) -- (v22)
                      (v1) -- (v11)
                      (v1) -- (v12)
                      ;

    \end{tikzpicture}
    \end{center}

\caption{%
    In the self-similar frog model on $\TT_d$, the initial frog moves from $\emptyset$ to $\emptyset'$
    to $v_1$ and then continues down the tree. The random variables $V$ and $V_1$, counting the number
    of frogs moving from $\emptyset'$ to $\emptyset$ and from $v_1$ to $\emptyset'$, respectively,
    both have distribution $\nu$ (see \thref{lem:fixed}). For $i\in\{2,\ldots,d\}$, the distribution of $V_i$
    conditional on a frog entering the subtree $\mathbb T_d(v_i)$ is also $\nu$.}
    \label{fig:1}
\end{figure}
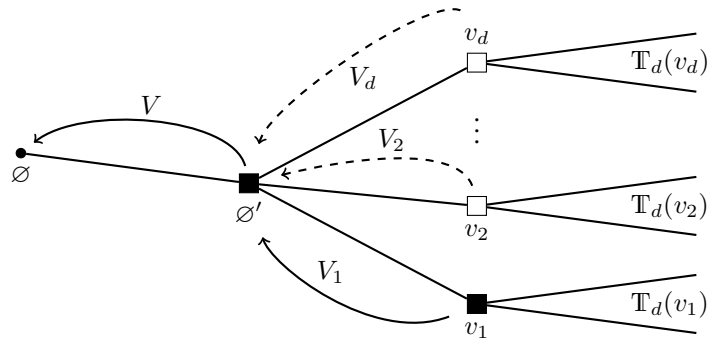


%


\subsubsection{The operators $\Bb$ and $\Uu$}\label{sec:operator}

Suppose that the initial frog in the self-similar frog model moves from $\emptyset$ to $\emptyset'$
to $v_1$. Let $v_2,\ldots,v_d$ be the remaining children of $\emptyset'$.
Observe that since frogs are stopped at $\emptyset$, no children
of $\emptyset$ other than $\emptyset'$ are ever visited.
The idea of this section is to view the self-similar frog model only at the vertices
mentioned above. If a vertex $v_i$ is visited, we close our eyes to $\TT_d(v_i)$, thinking of this
entire subtree as a black box that eventually emits some frogs from $v_i$ back to $\emptyset'$.

Enacting this view, we now define operators $\Bb$ and $\Uu$ on probability measures supported
on the extended nonnegative integers. Informally, the operator $\Bb$ corresponds to the number of visits to the root, and $\Uu$ corresponds to the number of subtrees $v_1,\hdots,v_d$ that are activated. Let $\pi$ be a probability measure on the nonnegative integers.
To define $\Bb\pi$ and $\Uu\pi$, we consider the following frog model.
The example to keep in mind is when $\pi$ is the law of $V$, in which case the following description
matches up with the black box view of the self-similar frog model described above.
\begin{description}
  \item[Graph] a star graph with center $\rho'$ and leaf vertices $\rho,u_1,\ldots,u_d$
    (think of these as paralleling $\emptyset'$ and $\emptyset,v_1,\ldots,v_d$). The root of the graph
    is $\rho$.
  \item[Sleeping frog counts] all independent, distributed as $\Poi(\mu)$
    at $\rho'$ and as $\pi$ at $u_1,\ldots,u_d$. There is one frog at $\rho$, as is always true
    at the root vertex.
  \item[Paths] All frogs have independent paths.
    The initial frog moves deterministically from $\rho$ to $\rho'$ to $u_1$ and then remains there.
    All other frogs, if woken, perform simple random non-backtracking walks from their starting points, stopped
    on arrival at a leaf vertex.
\end{description}
We then define two quantities:
\begin{itemize}
    \item[--] $\Bb\pi$ is the distribution of the number of frogs that terminate at $\rho$.
    \item[--] $\mathcal U \pi$ is the distribution of the final number of $u_1,\hdots, u_d$ that are visited by a frog.\vspace{-1pt}
\end{itemize}
Note that our definition of the initial frog path as deterministic is just for convenience.
By symmetry, we would arrive at the same measures $\Bb\pi$ and $\Uu\pi$
if it were also defined as a stopped simple random non-backtracking walk.

We mention that $\Bb$ is closely related to the operators $\Aa$
defined in \cite{HJJ1} and \cite{HJJ2}, but differs from both of them.
The operator $\Aa$ in \cite{HJJ1} is the same as $\Bb$ in the $d=2$ case if the initial distribution at
$\rho'$ in the definition of $\Bb$ is changed from $\Poi(\mu)$ to $\delta_1$  (except that $\Aa$ acts
on probability generating functions rather than distributions). The operator $\Aa$ in \cite{HJJ2}
is the same as $\Bb$ in the $d=2$ case. For $d\geq 3$, the two operators differ in that
in the frog model defining $\Aa$, frogs initially at $v_2,\ldots,v_d$
do not wake other frogs.

Now we relate this system back to the frog model.
\begin{lemma}\thlabel{lem:fixed}
Let $\nu$ be the law of $V=r(\eta,T)$, the number of visits to the root in the self-similar frog model on $\mathbb T_d$. It holds that $\mathcal B \nu = \nu$.
\end{lemma}

\begin{proof}
  Essentially, the frog model on the star graph exactly matches the black box view of the self-similar frog
  model described at the beginning of Section~\ref{sec:operator}, and the result then follows
  from \thref{fact:same,fact:independence}.
  To make this more formal, we couple the two frog models. We take full advantage of the abelian
  properties of the frog model by viewing the frogs' motions in a convenient order.

  Consider the frog model used to define $\Bb\nu$ as well as the self-similar frog
  model.
  We can couple the initial number of frogs on $\rho'$ to be the same as on
  $\emptyset'$,
  and we can couple the first (and only) step of each frog at $\rho'$ with the first step
  of the corresponding frog at~$\emptyset'$.\looseness=-1

  Let $V_i$ be the number of frogs that ever move from $v_i$ to $\emptyset'$ in the self-similar model,
  and let $U_i$ be the number of frogs initially at $u_i$ in the star graph model.
  Noting that $\TT_d(v_1)$ is activated by the initial frog, $V_1\sim\nu$ by \thref{fact:same}. By
  \thref{fact:independence}, $V_1$ is independent of all that we have coupled so far (that is, the number
  and first steps of frogs initially at $\emptyset'$). The random variable $U_1$ is also independent
  of all we have coupled so far and is distributed identically to $V_1$.
  We can therefore couple $U_1$ and $V_1$ to be equal.
  Next, we couple the second (and final) step
  of each frog at $u_1$ with the step of the corresponding frog counted by $V_1$ after it moves from
  $v_1$ to $\emptyset'$.

  Let $\Vset_1$ consist of the indices $i\in\{2,\ldots,d\}$ such that $u_i$ has been visited so far.
  By the construction of our coupling, we can also describe $\Vset_1$ as the set of
  $i\in\{2,\ldots,d\}$ such that $\TT_d(v_i)$ has been activated so far.
  Furthermore, identically many frogs have returned so far to $\rho$ as to $\emptyset$.
  By \thref{fact:same,fact:independence}, conditional on the information so far, the random variables
  $(V_i,\,i\in\Vset_1)$ are i.i.d.-$\nu$ and are independent of the information so far, as are
  the random variables $(U_1,\,i\in\Vset_1)$. We can therefore couple these two random vectors to
  be equal. We then couple the paths of the frogs at these vertices to match up as we did
  with the frogs at $u_1$ and $v_1$.

  As above,  a vertex $u_i$ is visited for the first time in this second round if and only if
  $\TT_d(v_i)$ is visited for the first time in this second round. Let $\Vset_2$ be the set of such $i$.
  We can repeat the coupling argument of the previous paragraph, maintaining identical numbers of frogs
  terminating at $\rho$ as at $\emptyset$, until we get an empty $\Vset_j$ and have counted all
  returns to $\rho$ and $\emptyset$. Thus, under this coupling,
  the number of frogs terminating at $\rho$ in the star graph model
  is the same as the number of frogs terminating at $\emptyset$
  in the self-similar model. The first of these counts has distribution~$\Bb\nu$,
  while the second has distribution~$\nu$, showing that the two are equal.
\end{proof}

The next lemma is similar to \cite[Lemma~10]{HJJ1} and \cite[Lemma~10]{HJJ2}.
\begin{lemma}\thlabel{lem:dom}
    If $\pi_1 \stprec \pi_2$, then $\mathcal B \pi_1 \stprec \mathcal B \pi_2$.
\end{lemma}

\begin{proof}
    This immediately follows from the coupling definition of stochastic dominance.
  We couple the frog models defining $\Bb\pi_1$ and $\Bb\pi_2$ so that the frogs in the former
  are a subset of the frogs of the latter model, resulting in more visits to the root.
\end{proof}

Just as in \cite[Lemma 11]{HJJ2}, the operator $\mathcal B $ applied to a Poisson distribution
yields a mixture of Poisson distributions. This is a consequence of the following property, known as
\emph{Poisson thinning}:
Consider a multinomial distribution with $\Poi(\lambda)$ trials and $n$-types, each having probability $p_k$.
Then the vector of outcomes is distributed as an independent collection of $\Poi(\lambda p_k)$ random variables.

 \begin{lemma} \thlabel{lem:Usd} Let $U$ be a random variable distributed as $\mathcal U \Poi(\lambda)$.
    \begin{align}\mathcal B \Poi(\lambda) =  \Poi\left (\f \mu {d+1} + U\f {\lambda }{d}  \right).\label{eq:Usd}
\end{align}
\end{lemma}

\begin{proof}
In the frog model defining $\Bb\Poi(\lambda)$, the number of frogs at $\rho'$ that move back to $\rho$ is distributed as $\Bin(\Poi(\mu),1/(d+1))$. By Poisson thinning, this is $\Poi(\mu/(d+1))$. Each visited $u_i$ releases $\Poi(\lambda)$ sleeping frogs. These will take a non-backtracking step back to $\rho'$, then with probability $1/d$ will move to $\rho$. Thus, each activated $u_i$ sends $\Poi(\lambda/d)$ frogs to $\rho$. It follows that
$$\mathcal B \Poi(\lambda) \sim \Poi(\mu / (d+1) ) + \textstyle \sum_{1}^{d} \ind{ u_i \text{ visited}} \Poi(\lambda/d).$$
The above sum is equal to $\sum_1^U \Poi(\lambda/d)$. By Poisson thinning, the $\Poi(\lambda/d)$ terms are independent of $U$. Applying additivity of Poisson random variables then brings us to the claimed formula.
\end{proof}

\subsection{Simplest proof of recurrence on the binary tree}\label{sec:simplest}

We now break from the main thread to give a short proof that the frog model on the binary tree with
Poisson-distributed frogs has a recurrence phase.
The idea of the argument is to use \thref{lem:Usd} to demonstrate that for
some $\delta>0$, it holds for all $\lambda\geq 0$ that $\Bb\Poi(\lambda)\stsucc\Poi(\lambda+\delta)$.
\thref{lem:fixed,lem:dom} then let us bootstrap out way to the conclusion that
 that $V$ is stochastically larger than any Poisson distribution, and hence $V=\infty$ a.s.

\begin{prop}\thlabel{prop:binary.recurrence}
  The frog model on $\TT_2$ with i.i.d.-$\Poi(\mu)$ frogs per site is recurrent for
  $\mu>3\log\bigl( (1+\sqrt{5})/2\bigr)\approx 1.4436$.
\end{prop}
\begin{proof}
  In the $d=2$ case, $\Bb\Poi(\lambda)$ is a particularly simple mixture:
  \begin{align}\label{eq:Bb.binary}
    \Bb\Poi(\lambda) =
    \begin{cases}
      \Poi\bigl( \frac{\mu}{3} + \frac{\lambda}{2}\bigr) & \text{with probability $\exp\bigl(-\frac{\mu}{3}-\frac{\lambda}{2}\bigr)$,}\\
      \Poi\bigl( \frac{\mu}{3} + \lambda\bigr) & \text{with probability $1-\exp\bigl(-\frac{\mu}{3}-\frac{\lambda}{2}\bigr)$.}
    \end{cases}
  \end{align}
  This follows from \thref{lem:Usd} once we show that
  \begin{align}\label{eq:Uu.binary}
    \Uu\Poi(\lambda) =
    \begin{cases}
      1 & \text{with probability $\exp\bigl(-\frac{\mu}{3}-\frac{\lambda}{2}\bigr)$,}\\
      2 & \text{with probability $1-\exp\bigl(-\frac{\mu}{3}-\frac{\lambda}{2}\bigr)$.}
    \end{cases}
  \end{align}
  Recall that $\Uu\Poi(\lambda)$ is the distribution of the number of vertices out of $\{u_1,u_2\}$
  that are visited in the frog model defining $\Bb$. The vertex $u_1$ is always visited by the initial
  frog in this model. Each of the $\Poi(\mu)$ frogs initially at $\rho'$ has a $1/3$ chance of visiting
  $u_2$, and each of the $\Poi(\lambda)$ frogs initially at $u_1$ has a $1/2$ chance of visiting $u_2$,
  all independently of each other. By Poisson thinning, the number of these frogs that visit $u_2$
  is distributed as $\Poi(\mu/3 + \lambda/2)$, and thus $u_2$ is visited with
  probability $1-\exp(-\mu/3-\lambda/2)$. This establishes \eqref{eq:Uu.binary} and hence
  \eqref{eq:Bb.binary}.

  \thref{thm:dominance.criterion} now instructs us to compute the probability placed on $0$
  by $\Bb\Poi(\lambda)$, which is
  \begin{align*}
    -\log\biggl[\bigl( 1 - e^{-\frac{\mu}{3}-\frac{\lambda}{2}} \bigr)e^{-\frac{\mu}{3}-\lambda}
       + e^{-\frac{2\mu}{3}-\lambda}\biggr]
      &= \lambda +\frac{\mu}{3} - \log\biggl[  1 - e^{-\frac{\mu}{3}-\frac{\lambda}{2}}
                                + e^{-\frac{\mu}{3}}\biggr] \\
      &\geq \lambda + \frac{\mu}{3} - \log\biggl[  1
                                + e^{-\frac{\mu}{3}}\biggr] >\lambda+\delta
  \end{align*}
  for some $\delta>0$ depending on $\mu$ but not on $\lambda$, under our assumption that
  $\mu>3\log\bigl( (1+\sqrt{5})/2\bigr)$.
  By \thref{thm:dominance.criterion},
  \begin{align}\label{eq:binary.bootstrap}
    \Bb\Poi(\lambda)\stsucc\Poi(\lambda+\delta)
  \end{align}
  for any $\lambda\geq 0$.

  Now, we carry out the bootstrap. Recall that $\nu$ is the distribution of $r(\eta,T)$, the number
  of visits to the root in the nonbacktracking frog model on the binary tree.
  As $\nu\stsucc\Poi(0)$, \thref{lem:dom} shows that
  $\Bb\nu\stsucc\Bb\Poi(0)$, and so $\Bb\nu\stsucc\Poi(\delta)$ by \eqref{eq:binary.bootstrap}.
  But $\nu$ is a fixed point of $\Bb$ by \thref{lem:fixed}, implying that
  $\nu\stsucc\Poi(\delta)$.
  Repeating this argument of successively applying \thref{lem:dom}, \eqref{eq:binary.bootstrap},
  and \thref{lem:fixed}, we show that $\nu\stsucc\Poi(2\delta)$, and so on. Thus
  $\nu$ is stochastically larger than all Poisson distributions, which implies $\nu=\delta_\infty$.
  Finally, recalling that $r(\eta,S)$ is the number of visits to the root in the frog model and
  that $r(\eta,S)\stsucc\nu$ by \thref{fact:mono,fact:coupling}, we can
  conclude that $r(\eta,S)=\infty$ a.s.
\end{proof}

\subsection{A more complicated bootstrap}\label{sec:bootstrap}
We now give an argument along the same lines as our proof of \thref{prop:binary.recurrence},
proving recurrence for all $d$ under assumptions on $\mu$ that are optimal up to constant
factors, as shown by the lower bound in \thref{thm:critical}. This bound follows from
\cite[Proposition 15]{HJJ2}, which is proven by coupling the frog model
with a transient branching random walk. Our contribution here is the upper bound.

Recall that $\Uu\Poi(\lambda)$ is the distribution of the number of vertices $u_1,\ldots,u_d$ visited
in the frog model on the star graph defined in Section~\ref{sec:operator}.
The essential difference between our proofs of recurrence for $\mu=\Omega(d)$ here and
for $\mu=\Omega(d\log d)$ in \cite{HJJ2} is that here we give a better lower bound on $\Uu\Poi(\lambda)$.
For a fixed $\lambda\geq 0$, we define a lower bounding random variable $U'\in\{1,\ldots,d\}$ as follows.
Consider the frog model used to define $\Bb\Poi(\lambda)$ and $\Uu\Poi(\lambda)$,
and observe how many of $u_1,\ldots,u_d$ are visited by the $\Poi(\mu)$ frogs starting at
$\rho'$. If at least $\kd$ of these vertices are visited
for a yet to be determined constant $c$,
then arbitrarily choose $\kd$ of them and allow the frogs activated there the chance
to visit the remaining $d-\kd$ vertices.
If fewer than $\kd$ vertices are visited by the frogs at $\rho'$, then recall that
$u_1$ is guaranteed to be activated by the initial frog, and
just use the frogs at $u_1$ to try to activate the remaining vertices $u_2,\ldots,u_d$.
We define $U'$ as the number of vertices out of $u_1,\ldots,u_d$ activated in the end in this scheme.
This is summarized as follows:

Let $U_1'$ be the number of vertices $u_1,\ldots,u_d$ visited by the frogs initially at
$\rho'$.
\begin{description}
  \item[Case~1]{$U_1'\geq\kd$\\
      Arbitrarily choose $\kd$ of the vertices counted by $U_1'$ and denote them by
      $\Vset\subseteq\{u_1,\ldots,u_d\}$.
      Let $U'$ be the sum of $\kd$ and the number of the remaining $d-\kd$ vertices
      $\{u_1,\ldots,u_d\}\setminus\Vset$ visited by frogs starting in $\Vset$.}
  \item[Case~2]{$U_1'<\kd$\\
      Let $U'$ equal one plus the number of number of vertices $u_2,\ldots,u_d$ visited
      by frogs returning from $u_1$.}
\end{description}
As
$U'$ counts only a subset of the full collection of activated vertices, we have $U'\stprec \Uu\Poi(\lambda)$.

We now sketch the argument in more detail. Throughout the remainder of the paper,
we will assume that $\mu=C(d+1)$ with
$C$ a yet to be determined positive constant. In \thref{lem:case2.small},
we prove that Case~2 occurs with exponentially small probability as $d$ grows.
Next, in \thref{lem:U''} we give a very explicit definition of a random variable $U''$ satisfying
$U''\stprec U'\stprec \Uu\Poi(\lambda)$.
In \thref{lem:log}, we use this lower bound together with \thref{lem:Usd}
to prove that if $V\stsucc\Poi(\lambda)$, then $V\stsucc\Poi(\lambda+\delta)$ for some $\delta>0$.
The same iterative argument used in \thref{prop:binary.recurrence} then implies that $V=\infty$ a.s.

\begin{lemma}\thlabel{lem:case2.small}
  Recall that $U_1'$ is the number of vertices $u_1,\ldots,v_d$ visited by the
  $\Poi(\mu)$ frogs initially at $\rho'$ in the frog model defining $\Bb\Poi(\lambda)$
  and $\Uu\Poi(\lambda)$. We have
  \begin{align}
    \P[U_1' < \kd ] \leq e^{ -bd }:=p, \label{eq:HD}
  \end{align}
  where $b = 2 (1 - e^{-C} - \f 1c)^2$.
\end{lemma}
\begin{proof}
It is a consequence of Poisson thinning that out of the $\Poi(\mu)$ frogs starting at $\rho'$,
independently $\Poi(\f \mu { d+1} ) = \Poi(C)$ move to each leaf $u_1,\ldots,u_d$.
Thus each vertex has an independent $1-e^{-C}$ chance of having a frog visit it from the ones
starting at $\rho'$, showing that $U_1' \sim \Bin(d,\,1-e^{-C})$.

Hoeffding's inequality tailored to a binomial distribution states that $\P[\Bin(n,p) \leq (p-\epsilon)n ]\leq \exp ( - 2 \epsilon^2 n)$ (this follows from \cite[eq.~(2.3)]{hoeffding}).
If we apply the inequality to $U_1'$ with $\epsilon =  (1 - e^{-C}) -\f 1c $, we establish
\eqref{eq:HD}.
\end{proof}

\begin{lemma}\thlabel{lem:U''}
  Let
  \begin{align}
    U'' \sim \begin{cases} \kd + \Bin\Bigl(d - \kd ,\, 1 - e^{ - \lambda /c}\Bigr) &\text{with probability $1-q$,}\\
                        1+ \Bin\Bigl(d-1,\, 1- e^{- \lambda/d} \Bigr)& \text{with probability $q$,}
  \end{cases}\label{eq:U''}
  \end{align}
  where $q=\P\bigl[U_1'<\kd\bigr]$. Then $U''\stprec U'$.
\end{lemma}
\begin{proof}
  Writing $U'\mid E$ to mean $U'$ conditioned on the event~$E$, we claim that
  \begin{align}\label{eq:case1}
    U'\mid\{U_1'\geq\kd\}\stsucc \kd + \Bin\Bigl(d - \kd ,\, 1 - e^{ - \lambda /c}\Bigr),
  \end{align}
  and
  \begin{align}\label{eq:case2}
    U'\mid\{U_1'<\kd\}\stsucc 1+ \Bin\Bigl(d-1,\, 1- e^{- \lambda/d} \Bigr).
  \end{align}
  The lemma then follows because conditional stochastic dominance implies
  stochastic dominance \cite[Theorem~1.A.3, (d)]{SS}.

  Thus it just remains to confirm \eqref{eq:case1} and \eqref{eq:case2}.
  Suppose $U_1'\geq\kd$. Then we are in Case~1, and $U'=\kd + U_2'$,
  where $U_2'$ is the number of vertices
  in $\{u_1,\ldots,u_d\}\setminus\Vset$ visited by frogs returning from $\Vset$.
  Conditional on $\Vset$, the counts of frogs proceeding from $\Vset$ to each of
  $\{u_1,\ldots,u_d\}\setminus\Vset$ form a collection of independent $\Poi(\lambda\kd/d)$ random variables.
  Thus each vertex in $\{u_1,\ldots,u_d\}\setminus\Vset$ has an independent probability
  of $1-e^{-\lambda\kd/d}\geq 1-e^{-\lambda/c}$ of being visited by one of these frogs, showing that
  $U_2'\stsucc\Bin(d-\kd,1-e^{-\lambda/c})$ and confirming \eqref{eq:case1}.

  Next, suppose that $U_1'<\kd$, and Case~2 is in effect.
  In this case, $U'=1+U_2'$, where $U_2'$ is the number of
  vertices $u_2,\ldots,u_d$ visited by frogs returning from $u_1$.
  By the same reasoning as in the previous case, $U_2'\stsucc\Bin(d-1,1-e^{-\lambda/d})$,
  confirming \eqref{eq:case2}.
\end{proof}

\begin{lemma}\thlabel{lem:log}
  Define
  $$\fC = h_{C,c}(\lambda ,d) := \log \Bigl[\bigl(e^{-\frac{\lambda}{c}+\frac{\lambda}{d}}+1-e^{-\frac{\lambda}{c}}\bigr)^{d-\kd}
     + p \bigl( 2-e^{-\frac{\lambda}{d}} \bigr)^{d-1}\Bigr],$$
  where $p$ is the value defined in \eqref{eq:HD}, which depends on $C$ and $c$.
  We have
  \begin{align*}
    \Bb\Poi(\lambda)\stsucc\Poi\biggl( \lambda +\frac{\mu}{d+1} - \fC \biggr).
  \end{align*}
\end{lemma}
\begin{proof}
  Combining \eqref{eq:Usd} and $U''\stprec \Uu\Poi(\lambda)$,
  it follows from \cite[Theorem~1.A.3, (d)]{SS} that
  \begin{align}\label{eq:interst}
    \Bb\Poi(\lambda)\stsucc \Poi\biggl (\f \mu {d+1} +  U''\f {\lambda }{d} \biggr).
  \end{align}
  In light of \thref{thm:dominance.criterion}, we would like to compute $- \log \E e^{- \f \lambda d U''}$.
  Recalling the definition of $U''$ in \eqref{eq:U''}, we use
  the fact that $\E x^{\Bin(n,p)}= ( 1- p + px)^n$ to compute
  \begin{align*}
    \E e^{ -\f \lambda d  U ''}&=(1-q)e^{-\frac{\lambda}{d}\kd}\Bigl(e^{-\frac{\lambda}{c}}+\bigl(1-e^{-\frac{\lambda}{c}}\bigr)e^{-\frac{\lambda}{d}}\Bigr)^{d-\kd}\\
    &\qquad\qquad + q e^{-\frac{\lambda}{d}}\Bigl( e^{-\frac{\lambda}{d}} +
                                             \bigl(1-e^{-\frac{\lambda}{d}}\bigr)e^{-\frac{\lambda}{d}} \Bigr)^{d-1}.
  \end{align*}
  Using the bound $q\leq p$ from \thref{lem:case2.small} and the trivial bound $1-q\leq 1$ in the first
  step, and factoring out $e^{-\lambda}$ in the second step,
  \begin{align*}
    \E e^{ -\f \lambda d  U ''}
     &\leq e^{-\frac{\lambda}{d}\kd}\Bigl(e^{-\frac{\lambda}{c}}+\bigl(1-e^{-\frac{\lambda}{c}}\bigr)e^{-\frac{\lambda}{d}}\Bigr)^{d-\kd}\\
    &\qquad\qquad + p e^{-\frac{\lambda}{d}}\Bigl( e^{-\frac{\lambda}{d}} +
                                             \bigl(1-e^{-\frac{\lambda}{d}}\bigr)e^{-\frac{\lambda}{d}} \Bigr)^{d-1}\\
    &= e^{-\lambda}\Bigl[\bigl(e^{-\frac{\lambda}{c}+\frac{\lambda}{d}}+1-e^{-\frac{\lambda}{c}}\bigr)^{d-\kd}
     + p \bigl( 2-e^{-\frac{\lambda}{d}} \bigr)^{d-1}\Bigr].
  \end{align*}
Thus,
$$- \log \E e^{ -\f \lambda d  U'' } = \lambda - h_{C,c}.
$$
Using the above calculation and \thref{thm:dominance.criterion}, we deduce that
$$\Poi\left(\f \mu {d+1} + U'' \f \lambda d\right)  \succeq\Poi\biggl( \lambda + \frac{\mu}{d+1} - \fC \biggr).$$
Together with \eqref{eq:interst}, this completes the proof.
\end{proof}

\begin{proof}[Proof of \thref{thm:critical}]
  As we noted before, the lower bound is a consequence of \cite[Proposition 15]{HJJ2}, and
  we just need to establish the upper bound by showing that the frog model on $\TT_d$ is almost surely
  recurrent with i.i.d.-$\Poi(2.28d)$ frogs per vertex for sufficiently large $d$.
  To apply our bootstrapping argument, we seek to show that for some $\delta>0$, it
  holds for all $\lambda\geq 0$ that $\Bb\Poi(\lambda)\stsucc\Poi(\lambda+\delta)$.
  Considering the result of \thref{lem:log},
  we need to choose $C$ and $c$ such that
  $\mu/(d+1)-\fC(\lambda,d)>\delta$ for all $\lambda\geq 0$ and sufficiently large~$d$. Recalling that
  $\mu=C(d+1)$, rearranging terms, and exponentiating both sides of the inequality,
  this is equivalent to showing that for some $C$, $c$, $\delta$, and $d_0$ it holds that
  \begin{align}
    \exp(\fC(\lambda,d)) &< e^{C-\delta},\label{eq:to.prove}
  \end{align}
  on the set $\{(\lambda, d)\colon \lambda\geq 0, d\geq d_0\}$.

  Towards proving this, we start with the inequality
  \begin{align}
    \exp(\fC(\lambda,d)) &= \bigl(e^{-\frac{\lambda}{c}+\frac{\lambda}{d}}+1-e^{-\frac{\lambda}{c}}\bigr)^{d-\kd}
     + p \bigl( 2-e^{-\frac{\lambda}{d}} \bigr)^{d-1}\nonumber\\
     &\leq \bigl(1 + e^{-\frac{\lambda}{c}}\bigl(e^{\frac{\lambda}{d}}-1\bigr) \bigr)^{d(1-\frac{1}{c})} + e^{-bd} 2^{ d-1}
     \label{eq:tobound}
  \end{align}
  obtained by applying the bounds $2-e^{-\lambda/d}\leq 2$ and $d- \kd  \leq d(1-1/c)$
  and substituting the value of $p$ from \eqref{eq:HD}.
  Note that $b$ depends on $C$ and $c$.
  Now we bound each of the two terms on the right hand side of \eqref{eq:tobound}
  for the right choice of $C$, $c$, and $d_0$.

  Some calculus shows that for any $d$ and $c$ satisfying $d>c$, the first term is maximized in
  $\lambda$ when $e^{\lambda/d}=d/(d-c)$. This demonstrates that if $d>c$, then
  \begin{align*}
    \bigl(1 + e^{-\frac{\lambda}{c}}\bigl(e^{\frac{\lambda}{d}}-1\bigr) \bigr)^{d(1-\frac{1}{c})}
      &\leq \Biggl( 1 + \biggl(\frac{d-c}{d}\biggr)^{d/c}\biggl(\frac{d}{d-c}-1\biggr)
        \Biggr)^{d(1-\frac{1}{c})}\\
      &= \Biggl( 1 + \biggl(\frac{d-c}{d}\biggr)^{d/c}\frac{c}{d-c} \Biggr)^{d(1-\frac{1}{c})}\\
    &\leq \biggl( 1 + \frac{c}{d-c} \biggr)^{d(1-\frac{1}{c})}\\
    &\leq \exp\biggl(\frac{d(c-1)}{d-c} \biggr).
  \end{align*}
  Thus, for any $\epsilon>0$, we can choose $d_0$ sufficiently large that for all $d\geq d_0$ and $\lambda>0$,
  \begin{align}\label{eq:term1}
    \bigl(1 + e^{-\frac{\lambda}{c}}\bigl(e^{\frac{\lambda}{d}}-1\bigr) \bigr)^{d(1-\frac{1}{c})}
      &\leq e^{c-1+\epsilon}.
  \end{align}
  The second term to be bounded, $e^{-bd}2^{d-1}$, vanishes as $d\to\infty$ when $b > \log 2$.
  Referring back to \eqref{eq:HD} and doing some algebra, we see that $b>\log 2$ when
  \begin{align}\label{eq:Cc}
    C > -\log\biggl(1 - \frac{1}{c}-\sqrt{\frac{\log 2}{2}}\biggr).
  \end{align}
  If this inequality holds, then
  for any $\epsilon>0$, we can choose $d_0$ large enough that for all $d\geq d_0$,
  \begin{align}\label{eq:term2}
    e^{-bd}2^{d-1}&\leq\epsilon.
  \end{align}

  Finally, we set $c=3.26$ and $C=2.27$, which satisfies \eqref{eq:Cc}.
  Applying \eqref{eq:term1} and \eqref{eq:term2}, for any $\epsilon>0$, there exists $d_0$
  such that for all $d\geq d_0$ and $\lambda\geq 0$,
  \begin{align*}
    \exp(\fC(\lambda,d)) &\leq e^{2.26+\epsilon} + \epsilon.
  \end{align*}
  Choosing $\epsilon,\delta>0$ sufficiently small, we can bound this by $e^{2.27-\delta}$,
  confirming \eqref{eq:to.prove}.

  In all, we have shown that for $d\geq d_0$, if $\mu\geq 2.27(d+1)$
  then for all $\lambda\geq 0$,
  \begin{align}\label{eq:bootstrap}
    \Bb\Poi(\lambda)\stsucc\Poi(\lambda+\delta).
  \end{align}
  Increasing $d_0$ as necessary, we can revise our assumption to $\mu\geq 2.28d$ for $d\geq d_0$.
  The rest of the proof is to use this to bootstrap our way to the conclusion that the number of visits
  to the root in the frog model is almost surely infinite given these assumptions, which proceeds
  identically as the last paragraph of \thref{prop:binary.recurrence}.
\end{proof}


\bibliographystyle{amsalpha}

\providecommand{\bysame}{\leavevmode\hbox to3em{\hrulefill}\thinspace}
\providecommand{\MR}{\relax\ifhmode\unskip\space\fi MR }
\providecommand{\MRhref}[2]{%
  \href{http://www.ams.org/mathscinet-getitem?mr=#1}{#2}
}
\providecommand{\href}[2]{#2}






\end{document}